\documentclass[12pt, 
]{amsart}

\usepackage{amssymb}
\usepackage{amsmath}
\usepackage{amsthm}
\usepackage{amscd}
\usepackage{graphicx}
\usepackage[all]{xy}

\newtheorem{theo}{Theorem}[section]

\newtheorem{lemm}[theo]{Lemma}
\newtheorem{cor}[theo]{Corollary}

\theoremstyle{definition}

\theoremstyle{remark}
\newtheorem{rem}[theo]{Remark}
\theoremstyle{definition}

\newcommand{\deldel}{\sqrt{-1}\partial \overline{\partial}}
\newcommand{\dbar}{\overline{\partial}}

\newcommand{\HH}{\mathcal{H}}
\newcommand{\A}{\alpha}
\newcommand{\R}{{\widetilde{\rho}}}

\newcommand{\C}[2]{{#1}_{i_{0}...i_{#2}}}
\newcommand{\CC}[2]{\{{#1}_{i_{0}...i_{#2}}\}}
\newcommand{\U}[1]{{U}_{i_{0}...i_{#1}}}
\newcommand{\V}[1]{{V}_{i_{0}...i_{#1}}}

\newcommand{\lla}[0]{{\langle\!\hspace{0.02cm} \!\langle}}
\newcommand{\rra}[0]{{\rangle\!\hspace{0.02cm}\!\rangle}}

\begin{document}

\title[A vanishing theorem of Koll\'ar-Ohsawa type]
{A vanishing theorem of Koll\'ar-Ohsawa type}

\author{Shin-ichi MATSUMURA}

\address{Mathematical Institute, Tohoku University, 
6-3, Aramaki Aza-Aoba, Aoba-ku, Sendai 980-8578, Japan.}

\email{{\tt mshinichi-math@tohoku.ac.jp, mshinichi0@gmail.com}}

\thanks{Classification AMS 2010:Primary 32L20, Secondary 14C30, 58A14. }


\maketitle

\begin{abstract}
For proper surjective holomorphic maps from K\"ahler manifolds to analytic spaces, 
we give a decomposition theorem for 
the cohomology groups of 
the canonical bundle twisted by Nakano semi-positive vector bundles 
by means of the higher direct image sheaves, 
by using the theory of harmonic integrals developed by Takegoshi. 
As an application, 
we prove a vanishing theorem of Koll\'ar-Ohsawa type 
by combining the $L^2$-method for the $\dbar$-equation. 
\end{abstract}

\section{Introduction}
The study of the direct image sheaves of adjoint bundles  
is one of important subjects   
in algebraic geometry and the theory of several complex variables. 
The main purpose of this paper is 
to generalize results related to this subject  
in \cite{Kol86}, \cite{Kol86b}, \cite{Ohs84}, \cite{Tak95}. 
In this paper, 
for a proper surjective holomorphic map $f \colon X \to Y$ 
from a K\"ahler manifold $X$ to an analytic space $Y$, 
we consider   
the higher direct image sheaves $R^q f_{*}(K_{X}\otimes E)$  
of the canonical bundle $K_{X}$ twisted 
by a  vector bundle $E$  
and the Leray spectral sequence 
$$
H^{p}(Y, R^{q}f_{*}(K_{X}\otimes E) ) \Longrightarrow H^{p+q}(X, K_{X}\otimes E). 
$$
In his paper \cite{Tak95}, 
Takegoshi proved  
the degeneration of the Leray spectral sequence 
at $E_{2}$-term and an injectivity theorem 
for the case $p=0$ 
when $E$ admits a (hermitian) metric 
with semi-positive curvature in the sense of Nakano. 
(See \cite{EV92}, \cite{Kol86}, \cite{Kol86b}, \cite{Sko78} 
for related topics, 
and see \cite{Fuj15}, \cite{Mat15b} 
for recent developments.)  
In order to generalize these results,  
we study the Leray spectral sequence in detail,  
by using the theory of harmonic integrals developed by Takegoshi. 
As a result, 
we obtain a decomposition theorem and 
an injectivity theorem for the case $p>0$ 
(Theorem \ref{m1} and Corollary \ref{m1-co}).  
Theorem \ref{m1} is stronger than the degeneration at $E_{2}$-term,  
and Corollary \ref{m1-co} is a generalization of 
the corollary of the main result in  \cite{Kol86b}. 
Moreover, as an application, we prove a generalization of the vanishing theorem of 
Koll\'ar-Ohsawa type (Theorem \ref{m2}), 
which gives an affirmative answer for \cite[Conjecture 2.25]{Fuj15}. 
The following theorem can be seen as 
a weak form of the decomposition theorem.

\begin{theo}\label{m1}
Let $f \colon X \to Y$ be a proper surjective holomorphic map  
from a K\"ahler manifold $X$ to an analytic space $Y$, 
and $E$ be a vector bundle on $X$ admitting a $($hermitian$)$ metric 
with semi-positive curvature in the sense of Nakano. 
Then, for integers $p, q \geq 0$, there exists a natural homomorphism
\begin{equation*}
\varphi_{p,q} \colon H^{p}(Y, R^{q}f_{*}(K_{X}\otimes E)) 
\to H^{p+q}(X, K_{X}\otimes E)
\end{equation*}
with the following properties\,$:$ 
\begin{enumerate}
\item[$\bullet$] The homomorphism $\varphi_{p,q}$ is injective. 
\item[$\bullet$] ${\rm{Im}}\, \varphi_{p, q} \cap {\rm{Im}}\, \varphi_{p',q'} = \{0\}$ 
when $p \not = p'$ and $p + q=p' + q'$. 
\end{enumerate}
Here $K_{X}$ denotes the canonical bundle of $X$, 
$R^q f_{*}(\bullet)$ denotes 
the $q$-th higher direct image sheaf, 
and ${\rm{Im}}\, \varphi_{p,q}$ denotes the image of $\varphi_{p,q}$.  
\end{theo}

When $X$ is a compact K\"ahler manifold, 
we obtain a  decomposition theorem 
for $H^{\ell}(X, K_{X}\otimes E)$
(the first conclusion of Corollary \ref{m1-co}). 
Moreover, we also obtain an injectivity theorem for the multiplication map induced 
by (holomorphic) sections of semi-positive line bundles 
(the latter conclusion of Corollary \ref{m1-co}).

\begin{cor}\label{m1-co}
Under the same situation as in Theorem \ref{m1}, 
we further assume that $X$ is a compact K\"ahler manifold. 
Then we have the following direct sum decomposition\,$:$ 
\begin{equation*}
H^{\ell}(X, K_{X}\otimes E)=\bigoplus_{p+q=\ell}\, {\rm{Im}}\, \varphi_{p,q}. 
\end{equation*}
Moreover, for a line bundle $F$ on $X$ admitting a metric with semi-positive curvature 
and a $($non-zero$)$ section $s$ of $F^m$ $(m > 0)$,  
the multiplication map induced by the tensor product with $s$ 
$$
\Phi_{s}\colon H^{p}(Y, R^q f_{*}(K_{X}\otimes E\otimes F)) 
\xrightarrow{\otimes s} 
H^{p}(Y, R^q f_{*}(K_{X}\otimes E\otimes F^{m+1})) 
$$
is injective for any $p, q \geq 0$.
\end{cor}

Note that the injectivity theorem for the case $p=0$ follows from 
\cite[Injectivity Theorem I\hspace{-.1em}I\hspace{-.1em}I]{Tak95}. 
In \cite{Fs15}, the above decomposition was  
generalized to a decomposition theorem in the derived category. 
By applying the proof of Theorem \ref{m1} (in particular, the construction of $\varphi_{p,q}$)  
and the $L^{2}$-method for the $\dbar$-equation, 
we can prove a vanishing theorem 
for the higher cohomology groups of  $R^q f_{*}(K_{X}\otimes E)$ 
if $E$ admits a metric $h$ satisfying 
$\sqrt{-1}\Theta_{h}(E) \geq f^{*}\sigma \otimes {\rm{id}}_{E}$  
for some K\"ahler form $\sigma$ on $Y$ in the sense of Nakano 
(see Section \ref{4} for the definition of K\"ahler forms on analytic spaces). 
When $X$ is compact and $q=0$, 
Ohsawa first proved such a vanishing theorem in \cite{Ohs84}, 
by giving the elegant method to solve the $\dbar$-equation on complete K\"ahler manifolds 
and inductively constructing solutions of the $\dbar$-equation. 
(See Remark \ref{compare} for a comparison between our proof and Ohsawa's proof.) 
In \cite{Kol86}, 
Koll\'ar proved the same conclusion 
when $X$ is a smooth projective variety and $E$ is the pull-back 
of an ample line bundle.
Therefore the following theorem can be seen as a generalization 
of Ohsawa's result to higher direct images 
and of Koll\'ar's result to the complex analytic setting.

\begin{theo}\label{m2}
Let $f \colon X \to Y$ be a proper surjective holomorphic map 
from a weakly pseudoconvex  K\"ahler manifold $X$ to an analytic space $Y$, and $E$ be a vector bundle on $X$.  
If $E$ admits a metric whose curvature is larger than or equal to 
the pull-back of a K\"ahler form on $Y$ in the sense of Nakano, 
then we have 
$$
H^{p}(Y, R^q f_{*}(K_{X}\otimes E)) =0 
\quad \text{for any } p>0 \text{ and } q\geq0. 
$$ 
In particular, we have the natural isomorphism 
$H^{0}(Y, R^q f_{*}(K_{X}\otimes E))\cong H^{q}(X, K_{X}\otimes E)$. 
\end{theo}

This paper is organized as follows\,$:$ 
In Section \ref{2}, we will briefly recall results  
on the theory of harmonic integrals developed in \cite{Tak95}. 
We will prove Theorem \ref{m1} and Corollary \ref{m1-co} in Section \ref{3}, 
and prove Theorem \ref{m2} in Section \ref{4}.

\subsection*{Acknowledgement}
The author would like to thank Professors  
Osamu Fujino and Taro Fujisawa for reading the draft and giving useful comments. 
He is supported by the Grant-in-Aid 
for Young Scientists (B) $\sharp$25800051 from JSPS.

\section{Preliminaries}\label{2}
In this section, we fix the notation used in this paper and 
summarize the theory of harmonic integrals developed in \cite{Tak95}.

Throughout this paper, 
let $f \colon X \to Y$ be a proper surjective holomorphic map  
from a K\"ahler manifold $X$ with a K\"ahler form $\omega$ to an analytic space $Y$, 
and $E$ be a vector bundle on $X$ with a hermitian metric $h$. 
Let $n$ be the dimension of $X$. 
For a  Stein open set $U$ in $Y$, 
we take an exhaustive smooth plurisubharmonic function 
$\Psi_{U}$ on $U$, 
and put $\Phi_{U}:=f^{*}\Psi_{U}$, 
which is also exhaustive and plurisubharmonic. 
Following Takegoshi's result in \cite{Tak95}, 
we define the space of harmonic forms on $X_{U}:=f^{-1}(U)$ by   
\begin{equation*}
\mathcal{H}^{n, k}(X_{U}, E, \Phi_{U})
=\{ u \in C_{\infty}^{n, k}(X_{U}, E) \mid 
\dbar u =0,  \dbar^{*}u=0, (\dbar \Phi_{U})^{*}u=0 \text{ on } X_{U}.  \}, 
\end{equation*}
where $\dbar^{*}$ (resp. ($\dbar \Phi_{U})^{*}$) is the adjoint operator 
of $\dbar$ (resp. the wedge product $\dbar \Phi_{U} \wedge \bullet$). 
This space can actually be shown to be independent of the choice of 
exhaustive plurisubharmonic functions if the curvature of $h$ is semi-positive 
(see \cite[Theorem 4.3 (ii)]{Tak95}). 
The following theorem plays a central 
role in the proof of Theorem \ref{m1}.  

\begin{theo}$($\cite[Theorem 4.3, 5.2]{Tak95}$).$
\label{Tak}
Under the same notation as above, 
we assume that the curvature of $h$ is semi-positive in the sense of Nakano. 
Then we have the following\,$:$\\
$(1)$ The natural quotient map from 
the space of smooth $\dbar$-closed $E$-valued $(n,k)$-forms 
to the $($Dolbeault$)$ cohomology group 
induces the isomorphism 
$$
\mathcal{H}^{n, k}(X_{U}, E, \Phi_{U}) \xrightarrow{\quad \cong \quad }
H^{k}(X_{U}, K_{X}\otimes E). 
$$
$(2)$ For Stein open sets $U_{1}, U_{2}$ in $Y$ such that 
$U_{2} \subset U_{1}$, 
the restriction map 
$$\mathcal{H}^{n, k}(X_{U_{1}}, E, \Phi_{U_{1}}) 
\longrightarrow \mathcal{H}^{n, k}(X_{U_{2}}, E, \Phi_{U_{2}})$$
is well-defined, and further it satisfies the following commutative diagram\,$:$
$$
\begin{CD}
\mathcal{H}^{n, k}(X_{U_{1}}, E, \Phi_{U_{1}})
@>\cong>>
H^{k}(X_{U_{1}}, K_{X}\otimes E) \\ 
@VVV  @VVV
\\  
\mathcal{H}^{n, k}(X_{U_{2}}, E, \Phi_{U_{2}})
@>\cong>>
H^{k}(X_{U_{2}}, K_{X}\otimes E). 
\end{CD}
$$
\end{theo}

In this paper, 
we denote by $\mathcal{H}_{U}$, 
the inverse map  
of the natural quotient map in the above theorem,  
namely 
$$
\mathcal{H}_{U}\colon 
H^{k}(X_{U}, K_{X}\otimes E)\xrightarrow{\quad \cong \quad }
\mathcal{H}^{n, k}(X_{U}, E, \Phi_{U}).  
$$
Further we often omit the subscript. 
To avoid  confusion in the proof of Theorem \ref{m1}, 
we should clearly distinguish 
between the (Dolbeault) cohomology class 
and the $\dbar$-closed $E$-valued form. 
For this reason, we denote the equality in 
$H^{k}(X_{U}, K_{X}\otimes E)$  by $\equiv$, and  
the cohomology class 
determined by a $\dbar$-closed $E$-valued form $\bullet$ by $[\bullet]$.  
We remark that we have 
$\alpha \equiv [\mathcal{H}_{U}(\alpha)]$  for 
an arbitrary cohomology class 
$\alpha \in H^{k}(X_{U}, K_{X}\otimes E)$.

In the proof of Theorem \ref{m2}, 
we consider the $L^{2}$-space of $E$-valued $(n, k)$-forms on $X$ with respect to 
$h$ and $\omega$ defined by   
\begin{equation*}
L_{(2)}^{n, k}(X, E)_{h, \omega}:= 
\{u \mid u \text{ is an }E\text{-valued }(n, k)\text{-form on } X \text{ with } 
\|u \|_{h, \omega}< \infty. \}. 
\end{equation*}
Here $\|u \|_{h, \omega}$ denotes the $L^{2}$-norm defined by 
\begin{equation*}
\|u \|^{2}_{h, \omega}:=\lla u, u \rra  _{h, \omega}:=
\int_{X} 
\langle u, u\rangle _{h, \omega}\, dV_{\omega}, 
\end{equation*}
where $\langle u, u \rangle  _{h, \omega}$ is the point-wise inner product with respect 
to $h, \omega$ and $dV_{\omega}$ is the volume form defined by $dV_{\omega}:=\omega^{n}/n!$. 
By the Bochner-Kodaira-Nakano identity, 
we obtain 
\begin{align*}
\lla  \sqrt{-1}\Theta_{h}(E) \Lambda_{\omega}\, v, v \rra 
\leq \| \dbar v \|^{2} + \| \dbar^{*} v \|^{2}  
\end{align*}
for compactly supported and smooth $E$-valued $(n,k)$-form $v$. 
Here $\Lambda_{\omega}$ is the (point-wise) adjoint operator of 
the wedge product $\omega \wedge \bullet$ with respect to 
$\langle \bullet, \bullet \rangle_{h,\omega}$.

\section{Decomposition by Cohomology Groups of Higher Direct Images}\label{3}
The purpose of this section is to prove  Theorem \ref{m1} 
and Corollary \ref{m1-co}. 
We will give the construction of the homomorphism $\varphi_{p,q}$ in subsection \ref{3-2}, 
by using the harmonic forms representing cohomology classes.  
Further we will show that $\varphi_{p,q}$ is injective 
in subsection \ref{3-3}  
and it gives a direct sum in subsection \ref{3-4}. 

\subsection{Set up}\label{3-1}
Throughout this section, 
we freely use the notation in the statement of the following theorem.

\begin{theo}[=Theorem \ref{m1}]
Let $f \colon X \to Y$ be a proper surjective holomorphic map  
from a K\"ahler manifold $X$ to an analytic space $Y$, 
and $E$ be a vector bundle on $X$ admitting $($hermitian$)$ metric 
with semi-positive curvature in the sense of Nakano. 
Then, for integers $p, q \geq 0$, there exists a natural homomorphism
\begin{equation*}
\varphi_{p,q} \colon H^{p}(Y, R^{q}f_{*}(K_{X}\otimes E)) 
\to H^{p+q}(X, K_{X}\otimes E)
\end{equation*}
with the following properties\,$:$ 
\begin{enumerate}
\item[$\bullet$] The homomorphism $\varphi_{p,q}$ is injective. 
\item[$\bullet$] ${\rm{Im}}\, \varphi_{p, q} \cap {\rm{Im}}\, \varphi_{p',q'} = \{0\}$ 
when $p \not = p'$ and $p + q=p' + q'$. 
\end{enumerate}
\end{theo}
We fix a K\"ahler form $\omega$ on $X$ and a metric $h$ on $E$ 
whose curvature is semi-positive in the sense of Nakano. 
By the assumption of the curvature of $h$, 
we can represent a cohomology class $\A$ 
by the associated harmonic form $\HH(\A)$ on $f^{-1}(U)$ 
for a Stein open set $U$ in $Y$.

\subsection{Construction of $\varphi_{p,q}$}\label{3-2}
First we take a Stein open cover $\mathcal{U}:=\{U_{i}\}_{i\in I}$ of $Y$, and 
consider the standard isomorphism 
$$
H^{p}(Y, R^{q}f_{*}(K_{X}\otimes E)) \cong 
\check{H}^{p}(\mathcal{U}, R^{q}f_{*}(K_{X}\otimes E)), 
$$
where the right hand side is 
the $\rm{\check{C}}$ech cohomology group calculated by $\mathcal{U}$. 
Since the open set $\U{p}:=U_{i_{0}}\cap...\cap U_{i_{p}}$ is also Stein, 
we have the natural isomorphism 
$$
H^{0}(\U{p}, R^{q}f_{*}(K_{X}\otimes E)) \xrightarrow{\quad \cong \quad}
H^{q}(f^{-1}(\U{p}), K_{X}\otimes E). 
$$
By this isomorphism, 
we identify 
the $p$-cochains valued in $R^{q}f_{*}(K_{X}\otimes E)$ 
with the $p$-cochains $\CC{\A}{p}_{i_{0}...i_{p}}$ valued in 
the (Dolbeault) cohomology of $E$-valued $(n,q)$-forms. 
It is sufficient for the construction of $\varphi_{p,q}$ 
to define $\varphi_{p,q}(\CC{\A}{p}_{i_{0}...i_{p}})$ 
for a given $p$-cochain $\CC{\A}{p}_{i_{0}...i_{p}}$ 
of the cohomology classes 
$\C{\A}{p} \in H^{q}(f^{-1}(\U{p}), K_{X}\otimes E)$ 
satisfying $\delta(\CC{\A}{p}_{i_{0}...i_{p}})\equiv 0$. 
Here $\equiv$ is the equality in the (Dolbeault) cohomology group  
and $\delta$ is the coboundary operator defined by
$$
\delta (\CC{\A}{p}_{i_{0}...i_{p}}):= 
\{\sum_{k=0}^{p+1} (-1)^{k} \A_{i_{0}...\hat{i_{k}}...i_{q+1}}  
|_{ f^{-1}(\U{p+1})} \}_{i_{0}...i_{p+1}}. 
$$
For simplicity, we put $\V{p}:=f^{-1}(\U{p})$, 
and further we omit the notation of the restriction in the right hand side
and the subscript, such as $\lq \lq i_{0}...i_{p}"$.

For an arbitrary $p$-cocycle $\CC{\A}{p}$ 
of $\C{\A}{p} \in H^{q} (\V{p}, K_{X}\otimes E)$, 
we will construct a $\dbar$-closed $E$-valued $(n, p+q)$-form on $X$ 
and define $\varphi_{p,q}(\CC{\A}{p}) \in H^{p+q}(X, K_{X}\otimes E)$
by its cohomology class.

First we consider the case $p=0$. 
In this case, by $\delta ( \{\A_{i_{0}}\}) \equiv 0$, 
we have
$$
\delta ( \{\HH (\A_{i_{0}})\}) 
=\HH(\delta ( \{\A_{i_{0}}\})) 
= 0. 
$$
Here we implicitly used the property (2) in Theorem \ref{Tak}. 
We remark that the above equality is 
the equality as the $E$-valued $(n,q)$-forms 
(not the cohomology classes). 
By the above equality, the family 
$\{\HH (\A_{i_{0}})\}$ is a $0$-cocycle of the $E$-valued $(n,q)$-forms, 
and thus it determines the $E$-valued $(n, q)$-form globally defined on $X$. 
Then we define $\varphi_{0,q}$ by $\varphi_{0,q}(\{\A_{i_{0}}\}):=[\{\HH (\A_{i_{0}})\}]$, 
where $[\bullet]$ denotes the cohomology class determined 
by a $\dbar$-closed $E$-valued form $\bullet$.

From now on, we consider the case $p>0$. 
For a given $p$-cocycle $\CC{\A}{p}$, 
by $\delta (\CC{\A}{p}) \equiv 0 $, 
we have  
$$
\delta ( \{ \HH (\C{\A}{p}) \} )
=\HH(\delta (\CC{\A}{p})) 
= 0. 
$$
Since 
the higher direct images of 
the \lq \lq smooth" $E$-valued $(n,q)$-forms are fine sheaves, 
there exists a $(p-1)$-cochain $\CC{b}{p-1}$ valued in 
the (not necessarily $\dbar$-closed) $E$-valued $(n,q)$-forms 
such that $\{ \HH (\C{\A}{p}) \}=\delta (\CC{b}{p-1})$. 
In fact, we can concretely construct $\CC{b}{p-1}$ 
by using a partition of unity associated to $\mathcal{U}$ 
from $\HH (\C{\A}{p})$(see Lemma \ref{par}). 
(The construction of $\CC{b}{p-1}$ 
plays an important role in the proof of Theorem \ref{m2}.)  
Since $\HH (\C{\A}{p})$ is $\dbar$-closed,  
we can easily see that 
$$
\delta( \{\dbar \C{b}{p-1} \})=\dbar \delta (\CC{b}{p-1})= \dbar\{ \HH (\C{\A}{p}) \}=0. 
$$
Then, by the same argument, 
there exists  $(p-2)$-cochain $\CC{b}{p-2}$  valued in 
the $E$-valued $(n,q+1)$-forms  
such that  $ \{\dbar \C{b}{p-1} \}=\delta (\CC{b}{p-2})$. 
Since $\dbar \C{b}{p-1}$ is also $\dbar$-closed, 
we obtain $\delta( \{\dbar \C{b}{p-2} \})=0$. 
Therefore there exists $(p-3)$-cochain $\CC{b}{p-3}$ valued in 
the $E$-valued $(n,q+2)$-forms 
such that $ \{\dbar \C{b}{p-2} \}=\delta (\CC{b}{p-3})$. 
By repeating this process, for $k=1, 2, \dots, p$,  
we can obtain $(p-k)$-cochain  $\CC{b}{p-k}$ 
valued in the $E$-valued $(n,q+k-1)$-forms  
satisfying the following equalities\,$:$ 

\[
  (*) \left\{ \quad
  \begin{array}{ll}
\vspace{0.2cm}
    \{ \HH (\C{\A}{p}) \}&=\delta (\CC{b}{p-1}),  \\
\{\dbar \C{b}{p-1} \}&=\delta (\CC{b}{p-2}),  \\
 & \vdots   \\
\{\dbar b_{i_{0}i_{1}} \}&=\delta (\{b_{i_{0}}\}).
  \end{array} \right.
\]
By the last equality, we can easily check $\delta (\{\dbar b_{i_{0}}\})=0$, 
which says that $\{\dbar b_{i_{0}}\}$ determines the $E$-valued 
$(n,p+q)$-form on $X$. 
We define $\varphi_{p,q}(\CC{\A}{p})$ by $\varphi_{p,q}(\CC{\A}{p}):=[\{\dbar b_{i_{0}}\}]$.

At the end of this subsection, 
we prove that the map $\varphi_{p,q}$ is well-defined. 
When $p$ is zero, the equality $\{\A_{i_{0}}\}\equiv \{\A'_{i_{0}}\}$ 
implies that $\HH(\A_{i_{0}})=\HH(\A'_{i_{0}})$ on $U_{i_{0}}$. 
Therefore $\varphi_{0,q}$ is well-defined. 
In the case $p>0$, for given  
$\CC{\A}{p}$ and $\CC{\A'}{p}$ such that 
$\CC{\A}{p} \equiv \CC{\A'}{p} + \delta (\CC{c}{p-1})$ 
for some $(p-1)$-cochain $\CC{c}{p-1}$,  
we take an arbitrary $\CC{b}{p-k}$ satisfying $(*)$ 
(resp. $\CC{b'}{p-k}$) satisfying $(*)$ 
for $\CC{\A}{p}$ (resp. $\CC{\A'}{p}$). 
Then there exists a $(p-2)$-cochain $\CC{d}{p-2}$ such that 
$$
\{ \C{b}{p-1} - \C{b'}{p-1}- \HH(\C{c}{p-1})\} 
= \delta( \{\C{d}{p-2} \} )
$$
since the left hand side is $\delta$-closed. 
Putting $\C{c}{p-2} = \C{ \dbar d}{p-2}$, 
we can take a $(p-3)$-cochain 
$\CC{d}{p-3}$ satisfying
$$
\{ \C{b}{p-2} - \C{b'}{p-2}- \C{c}{p-2}\} 
= \delta( \{\C{d}{p-3} \} ) 
$$
since the right hand side is $\delta$-closed. 
By putting $\C{c}{p-3} = \C{ \dbar d}{p-3}$ again, 
we can repeat this process, 
and thus we obtain 
$\C{c}{p-k} := \C{\dbar d}{p-k}$ for $k=1,2,\dots,p$. 
Then, by the same argument, 
we can easily check $ \delta \{b_{i_{0}} - b'_{i_{0}} - c_{i_{0}} \} =0$, 
which says that 
$\{b_{i_{0}} - b'_{i_{0}} - c_{i_{0}} \}$ determines the $E$-valued $(n,p+q-1)$-form $\eta$ on $X$. 
Then we have  
 $\{ \dbar b_{i_{0}}\} = \{\dbar b'_{i_{0}}\} + \dbar \eta$ 
since $c_{i_{0}}$ is $\dbar$-closed. 
Therefore $\{ \dbar b_{i_{0}}\} $ and $ \{\dbar b'_{i_{0}}\} $ 
give the same cohomology class. 
We can easily see that the map $\varphi_{p,q}$ 
does not depend on the choice of Stein open covers.

\subsection{Injectivity of $\varphi_{p,q}$}\label{3-3}
In this subsection, we will show that 
the map $\varphi_{p,q}$ 
constructed in subsection \ref{3-2} is injective, 
by using the theory of harmonic integrals again.

First we consider the case $p=0$. 
If we have $\varphi_{0,q}(\{\A_{i_{0}}\})\equiv 0$, 
then there exists an $E$-valued $(n,q-1)$-form $\eta$ on $X$  
such that $\{ \HH (\A_{i_{0}}) \} = \dbar \eta$. 
Since the map 
$$
\mathcal{H}\colon 
H^{q}(V_{i_{0}}, K_{X}\otimes E)\xrightarrow{\quad \cong \quad }
\mathcal{H}^{n, q}(V_{i_{0}}, E, \Phi_{U_{i_{0}}}).  
$$
is the inverse map of the natural quotient map (see Theorem \ref{Tak}), 
we have  $\A_{i_{0}} \equiv [ \HH (\A_{i_{0}}) ]$. 
Therefore we obtain  
$$
\A_{i_{0}} \equiv [ \HH (\A_{i_{0}}) ] \equiv [\dbar \eta] \equiv 0. 
$$

From now on, we consider the case $p>0$. 
We assume that $\varphi_{p,q}(\{ \C{\A}{p} \})\equiv 0$ for a $p$-cocycle $\{ \C{\A}{p} \}$. 
It is sufficient for the proof to construct a 
$(p-1)$-cochain $\CC{c}{p-1}$ valued in 
the $E$-valued $(n, q)$-forms 
satisfying the following properties\,$:$
$$
\bullet \ \   \CC{\A}{p}  \equiv \delta (\{[\C{c}{p-1}]\}).  
\quad \quad \quad 
\bullet \ \ \C{c}{p-1} \text{ is }  \dbar\text{-}\text{closed.}
$$
When we take a $(p-k)$-cochain $\CC{b}{p-k}$ satisfying $(*)$ 
for $k=1,2,\dots,p$, 
we have $\varphi_{p,q}(\{ \C{\A}{p} \})=[ \{ \dbar b_{i_{0}} \}  ] \equiv 0$. 
Therefore there exists an $E$-valued $(n,p+q-1)$-form $\eta$ on $X$
such that $\{ \dbar b_{i_{0}} \} = \dbar \eta$.
By putting $c_{i_{0}}:= b_{i_{0}} - \eta$, we have the following properties\,$:$ 
\begin{equation*}
\bullet \ \  \{ \dbar b_{i_{0}i_{1}} \}= \delta (\{c_{i_{0}}\}).  \quad \quad \quad 
\bullet \ \ c_{i_{0}} \text{ is }  \dbar\text{-}\text{closed.}
\end{equation*}
Then we can obtain  
$0=\delta (\{ \HH(c_{i_{0}}) \})$ since 
the left hand side in the above equality 
is $\dbar$-exact on $V_{i_{0}i_{1}}$.  
Further we can take an $E$-valued $(n,p+q-2)$-form 
$d_{i_{0}}$ such that 
$c_{i_{0}} = \HH(c_{i_{0}}) + \dbar d_{i_{0}} $
since $c_{i_{0}}$ and $\HH(c_{i_{0}})$ 
determine the same cohomology class.  
From the above argument, we can easily see that  
$$
\{ \dbar b_{i_{0}i_{1}} \} 
= \delta(\{ \HH(c_{i_{0}}) + \dbar d_{i_{0}}\}) 
= \delta(\{\dbar d_{i_{0}}\})=\dbar \delta (\{d_{i_{0}}\}). 
$$
Putting $\{c_{i_{0}i_{1}}\}:=\{b_{i_{0}i_{1}}\} - \delta (\{d_{i_{0}}\})$, 
we have the following properties\,$:$ 
\begin{equation*}
\bullet \ \  \{ \dbar b_{i_{0}i_{1}i_{2}} \}= \delta (\{c_{i_{0}i_{1}}\}).  \quad \quad \quad 
\bullet \ \ c_{i_{0}i_{1}} \text{ is }  \dbar\text{-}\text{closed.}
\end{equation*}
By repeating this process, 
we obtain a $(p-1)$-cochain $\CC{c}{p-1}$
valued in the $\dbar$-closed $E$-valued $(n, q)$-forms 
satisfying $\{ \HH(\C{\alpha}{p}) \} 
= \delta (\CC{c}{p-1})$. 
Then we can easily see that   
$$
\CC{\alpha}{p} \equiv \{ [\HH(\C{\alpha}{p})] \} \equiv \delta (\{[\C{c}{p-1}]\}). 
$$
This completes the proof of the injectivity.

\subsection{On the image of $\varphi_{p,q}$}\label{3-4}
In this subsection, we prove that 
${\rm{Im}}\, \varphi_{p,q} \cap {\rm{Im}}\, \varphi_{p',q'} = \{0\}$ 
when $p \not = p'$ and $p + q=p' + q'$.  
Without loss of generality, we may assume $p<p'$. 
For a $p$-cocycle $\CC{\alpha}{p}$ 
and a $p'$-cocycle $\CC{\alpha'}{p'}$ valued 
in the (Dolbeault) cohomology group, 
we assume that $\varphi_{p,q}(\CC{\alpha}{p}) = \varphi_{p',q'}(\CC{\alpha'}{p'})$. 
We take a $(p-k)$-cochain $\CC{b}{p-k}$ for $k=1,2,\dots, p$
(resp. a $(p'-k')$-cochain $\CC{b'}{p'-k'}$ for $k'=1,2,\dots, p'$) satisfying $(*)$. 
By the assumption, there exists an $E$-valued 
$(n,p+q-1)$-form $\eta$ on $X$ 
such that $\{\dbar b_{i_{0}} \} -\{ \dbar b'_{i_{0}}\}=\dbar \eta$. 
Then, by putting $c_{i_{0}}:=b_{i_{0}}-b'_{i_{0}}-\eta$, 
we have the same properties as in subsection \ref{3-3}\,$:$ 
\begin{equation*}
\bullet \ \  \{ \dbar (b_{i_{0}i_{1}} - b'_{i_{0}i_{1}}) \} 
= \delta (\{c_{i_{0}}\}).  \quad \quad \quad 
\bullet \ \ c_{i_{0}} \text{ is }  \dbar\text{-}\text{closed.}
\end{equation*}
By repeating the same argument as in subsection \ref{3-3}, 
we can construct a $(p-1)$-cochain $\CC{c}{p-1}$ valued in 
the $\dbar$-closed $E$-valued $(n,q)$-forms 
satisfying 
$\{ \HH(\C{\alpha}{p})-\dbar \C{b'}{p'} \}=\delta (\CC{c}{p-1}) $. 
Here we used the assumption of $p<p'$.  
By taking $\HH(\bullet)$, we obtain 
$\{ \HH(\C{\alpha}{p}) \}=\delta (\{ \HH(\C{c}{p-1}) \}) $, 
and thus we can show   
$$
\{ \C{\A}{p} \} \equiv \{[\HH(\C{\A}{p-1})]\}
\equiv 
\delta (\{ [\HH(\C{c}{p-1})] \}).  
$$ 
This completes the proof.

\subsection{Proof of Corollary \ref{m1-co}}\label{3-5}
In this subsection, we give a proof of Corollary \ref{m1-co}. 
Let $f \colon X \to Y$ be a proper surjective holomorphic map 
from a compact K\"ahler manifold $X$ to an analytic space $Y$. 
Since $E$ is assumed to admit a metric 
with semi-positive curvature in the sense of Nakano, 
we have 
\begin{equation*}
H^{\ell}(X, K_{X}\otimes E) \supset \bigoplus_{p+q=\ell}\, {\rm{Im}}\, \varphi_{p,q}. 
\end{equation*}
by Theorem \ref{m1}. 
Takegoshi proved that the Leray spectral sequence for $f$
$$
H^{p}(Y, R^{q}f_{*}(K_{X}\otimes E) ) \Longrightarrow H^{p+q}(X, K_{X}\otimes E)
$$
degenerates at $E_{2}$-term (see \cite[I Decomposition theorem ]{Tak95}). 
In particular, when $X$ is a compact K\"ahler manifold, 
$\dim H^{\ell}(X, K_{X}\otimes E)$ is equal to 
$\sum_{p+q=\ell}\dim H^{p}(Y, R^{q}f_{*}(K_{X}\otimes E) )$. 
This leads to the first conclusion of Corollary \ref{m1-co}.

Now we prove the latter conclusion of Corollary \ref{m1-co}. 
Let $F$ be a line bundle on $X$ admitting a metric 
with semi-positive curvature. 
The multiplication maps $\Phi_{s}$ and $\Psi_{s}$ 
induced by a section $s$ of $F^m$ satisfy  
the following commutative diagram\,$:$ 
$$
\begin{CD}
H^{p}(Y, R^q f_{*}(K_{X}\otimes E\otimes F))
@>\varphi_{p,q}>>
H^{p+q}(X, K_{X}\otimes E\otimes F) \\ 
@V\Phi_{s}VV  @V\Psi_{s}VV
\\  
H^{p}(Y, R^q f_{*}(K_{X}\otimes E\otimes F^{m+1})) 
@>\varphi'_{p,q}>>
H^{p+q}(X, K_{X}\otimes E\otimes F^{m+1})
\end{CD}
$$
Indeed, for a given 
$\CC{\alpha}{p} \in H^{p}(Y, R^q f_{*}(K_{X}\otimes E\otimes F))$, 
we can check $[s\mathcal{H}(\C{\alpha}{p})] \equiv s\C{\alpha}{p}$. 
Then we have $s\mathcal{H}(\C{\alpha}{p}) \equiv \mathcal{H}({s\C{\alpha}{p}})$
since $s\mathcal{H}(\C{\alpha}{p})$ is also harmonic 
by \cite[Proposition 4.4]{Tak95}. 
It implies that, for a $(p-k)$-cochain  $\CC{b}{p-k}$ satisfying $(*)$ for 
$\CC{\alpha}{p}$, 
the $(p-k)$-cochain $\CC{sb}{p-k}$ also satisfies 
$(*)$ for $\CC{s\alpha}{p}$. 
Therefore the above diagram is commutative.

Since $F$ admits a metric with semi-positive curvature, 
the vector bundles $E\otimes F$ and $E\otimes F^{m+1}$ are so.
The map $\varphi_{p,q}$ and $\varphi'_{p,q}$ are injective 
by Theorem \ref{m1}, 
and further $\Psi_{s}$ is also injective 
by the standard injectivity theorem (see \cite{Eno90}, \cite{Tak95}). 
Therefore the multiplication map $\Phi_{s}$ is also injective.

\section{Vanishing Theorem for Cohomology Groups 
of Higher Direct Images}\label{4}
In this section, we give a proof of Theorem \ref{m2}. 
The proof is a combination of Theorem \ref{m1} 
and techniques to solve the $\dbar$-equation with the $L^{2}$-estimate. 
\begin{theo}[= Theorem \ref{m2}]
Let $f \colon X \to Y$ be a proper surjective holomorphic map 
from a weakly pseudoconvex K\"ahler manifold $X$ to an analytic space $Y$ and $E$ be a vector bundle on $X$.  
If $E$ admits a metric $g$ satisfying  
$\sqrt{-1}\Theta_{g}(E) \geq f^{*}\sigma \otimes {\rm{id}}_{E}$ 
for some K\"ahler form $\sigma$ on $Y$ in the sense of Nakano,  
then we have 
$$
H^{p}(Y, R^q f_{*}(K_{X}\otimes E)) =0 
\quad \text{for any } p>0 \text{ and } q\geq0. 
$$  
\end{theo}

A K\"ahler form $\sigma$ on the regular locus $Y_{\rm{reg}}$ 
is said to be a K\"ahler form on $Y$ 
if for every point $y \in Y$ there exist a local embedding 
$i \colon U_{y} \hookrightarrow \mathbb{C}^{m}$ of 
an open neighborhood $ U_{y}$ of $y$
and a K\"ahler form $\widetilde{\sigma}$ on a neighborhood of $i( U_{y})$ 
such that $\sigma=i ^{*}\widetilde{\sigma}$ on $U_{y} \cap Y_{\rm{reg}}$. 
The pull-back $f^{*}\sigma$ is the extension of $f^{*}\sigma$ on $f^{-1}(Y_{\rm{reg}})$ 
to $X$, which is a semi-positive $(1,1)$-form on $X$. 
We show that the natural map $\varphi_{p, q}$ is the zero map 
under the assumption on the curvature of $E$. 
Then we obtain a vanishing theorem 
for the higher cohomology groups of $R^q f_{*}(K_{X}\otimes E)$ by Theorem \ref{m1}.

In the same way as in Section \ref{3}, 
we fix a Stein open cover $\mathcal{U}:=\{U_{i}\}_{i\in I}$ of $Y$, 
and compute the $\rm{\check{C}}$ech cohomology group. 
It is easy to prove the following lemma, 
which gives an explicit form of $\C{b}{p-k}$ satisfying $(*)$. 
The explicit form plays an important role 
when we solve the $\dbar$-equation. 
From now on, we fix a partition of unity $\{\rho_{i}\}_{i \in I}$
associated to $\mathcal{U}$, and put $\R_{i}:=f^{*}\rho_{i}$. 

\begin{lemm}\label{par}
Let $\CC{x}{\ell}$ be an $\ell$-cocycle valued in 
the $E$-valued $(n,q)$-forms. 
Then $\C{y}{\ell-1}:=\sum_{j \in I}\R_{j}\, x_{ji_{0}...i_{\ell-1}}$ 
satisfies 
$$
\CC{x}{\ell}=\delta (\CC{y}{\ell-1}). 
$$
\end{lemm}
The proof is a straightforward computation, and thus we omit it. 
Assume that $p>0$. 
For an arbitrary class in $\check{H}^{p}(\mathcal{U}, R^{q}f_{*}(K_{X}\otimes E))$, 
we take a $p$-cocycle $\CC{\A}{p} $ representing the class. 
By the above lemma,   
the $E$-valued form $\C{b}{p-k}$ defined inductively by 
$$
\C{b}{p-1}:=\sum_{j \in I}\R_{j}\, \mathcal{H}(\A_{ji_{0}...i_{p-1}}) 
\quad \text{and} \quad 
\C{b}{p-k}:=\sum_{j \in I}\R_{j}\, \dbar b_{ji_{0}...i_{p-k}}
$$
satisfies $(*)$. 
In particular, we have  
$$
\dbar b_{i_{0}}=\dbar \Big( \sum_{j \in I}\R_{j}\, \dbar b_{ji_{0}}\Big)
= \sum_{j \in I} \dbar \R_{j}\, \wedge \dbar b_{ji_{0}}
\quad \text{on }  U_{i_{0}}. 
$$
From the above equality, we will prove that 
the $E$-valued $(n,p+q)$-form 
$\varphi_{p,q}(\CC{\A}{p})=\{\dbar b_{i_{0}}\}$ is $\dbar$-exact on $X$, 
namely there exists an $E$-valued $(n,p+q-1)$-form $\eta$ on $X$ 
such that $\{\dbar b_{i_{0}}\} = \dbar \eta $. 
It implies that the map $\varphi_{p,q}$ is the zero map.

Fix a complete K\"ahler form $\omega$ on $X$. 
Note that $X$ admits a complete K\"ahler form 
since $X$ is a weakly pseudoconvex K\"ahler manifold. 
We take a metric $g$ on $E$ 
such that $\sqrt{-1}\Theta_{g}(E) \geq  f^{*}\sigma \otimes {\rm{id}}_{E}$ holds 
for some K\"ahler form $\sigma$ in the sense of Nakano. 
Further we take an exhaustive smooth plurisubharmonic function 
$\Phi$ on $X$. 
To solve the $\dbar$-equation $\{\dbar b_{i_{0}}\} = \dbar \eta $ 
by the standard technique (cf. \cite{Dem82}, \cite{Ohs84}), 
we consider the new metric $h$ on $E$ defined by 
$$
h:= g e^{- 2 \chi \circ \Phi}. 
$$
Here $\chi \colon \mathbb{R} \to \mathbb{R}$ 
is an increasing convex function, 
which will be suitably chosen later. 
The composite function $\chi \circ \Phi$ 
is also plurisubharmonic, 
and thus we have 
$$
\sqrt{-1}\Theta_{h}(E) =
\sqrt{-1}\Theta_{g}(E) + (2\deldel \chi \circ \Phi)\otimes 
{\rm{Id}}_{E}
\geq_{\rm{Nak}}   f^{*}\sigma \otimes {\rm{id}}_{E}. 
$$ 
From now on, we mainly handle the norm 
with respect to 
$h$ (not $g$ ) and $\omega$. 
(We often omit the subscript.)
We consider
the linear functional 
\begin{align*}
{\rm{Im}}\, \dbar^{*} \longrightarrow \mathbb{C} 
\quad \text{defined by} 
\quad w=\dbar^{*}v \longmapsto \lla v, \{\dbar b_{i_{0}}\} \rra_{h, \omega}, 
\end{align*}
where ${\rm{Im}}\, \dbar^{*}$ is 
the range of the closed operator $\dbar^{*}$ 
in the $L^{2}$-space $L_{(2)}^{n, p+q-1}(X, E)_{h, \omega}$. 
For the proof, it is sufficient to prove that 
the above linear map is (well-defined and) bounded, 
that is, there exists a positive constant $C$ such that 
$$
| \lla v, \{\dbar b_{i_{0}}\} \rra |^{2} \leq C \|\dbar^{*} v \|^{2} \quad 
\text{ for any } v \in {\rm{Dom}}\, \dbar^{*}.  
$$ 
Indeed, we can obtain 
an $E$-valued $(n,p+q-1)$-form $\eta$ on $X$ 
such that $\lla v, \{\dbar b_{i_{0}}\}\rra=\lla \dbar^{*} v, \eta  \rra$ 
for any $v \in {\rm{Dom}} \, \dbar^{*}$ and $\|\eta \|^2 \leq C$, 
by the Hahn-Banach theorem and the Riesz representation theorem. 
It gives a solution 
of the $\dbar$-equation $\{\dbar b_{i_{0}}\}=\dbar \eta $.

It is sufficient for the above estimate 
to prove that there exists a positive constant $C>0$ such that 
$$
|\lla v, \{\dbar b_{i_{0}}\}\rra|^{2} \leq 
C (\| \dbar v \|^{2} + \| \dbar^{*} v \|^{2})
$$
for compactly supported smooth $v$. 
From this inequality, 
we know that the above inequality also holds 
for arbitrary $v \in  {\rm{Dom}} \, \dbar \cap {\rm{Dom}} \, \dbar^{*}$.   
This is because, for a given $v \in  {\rm{Dom}} \, \dbar \cap {\rm{Dom}} \, \dbar^{*}$, 
we can take a compactly supported smooth $v_{k}$ 
such that $v_{k} \to v$, $\dbar v_{k} \to \dbar v$, and $\dbar^{*} v_{k} \to \dbar^{*} v$ 
in the $L^{2}$-space $L_{(2)}^{n, \bullet}(X, E)_{h, \omega}$ 
by the density lemma 
(for example, see 
\cite[(3.2)\,Theorem, ChapterV\hspace{-.1em}I\hspace{-.1em}I\hspace{-.1em}I]{Dem-book}). 
Here we used the condition that $\omega$ is complete. 
Then, for any $v \in {\rm{Dom}}\, \dbar^{*} $, we can easily see that 
$$
|\lla v, \{\dbar b_{i_{0}}\}\rra|^{2} = 
|\lla v_{1}, \{\dbar b_{i_{0}}\}\rra|^{2}
\leq 
C \| \dbar^{*} v_{1} \|^{2}
\leq 
C \| \dbar^{*} v \|^{2}
$$
by the orthogonal decomposition   
$$
v=v_{1}+v_{2}\in {\rm{Ker}}\, \dbar \, \oplus 
({\rm{Ker}}\, \dbar)^\perp. 
$$

To prove the above inequality for compactly supported smooth $v$, we consider 
the operator $B_{\delta}:=\omega_{\delta} \Lambda_{\omega} $ 
acting on the $L^{2}$-space $L_{(2)}^{n, \bullet}(X, E)_{h, \omega}$, 
where $\omega_{\delta}$ is the K\"ahler form on $X$ defined by 
$\omega_{\delta}=\delta \omega + f^{*}\sigma$ and 
$\Lambda_{\omega}$ is the (point-wise) adjoint operator of 
the wedge product $\omega \wedge \bullet$ with respect to 
$\langle \bullet, \bullet \rangle_{h,\omega}$. 
The operator $B_{\delta}$ is positive 
definite (for example see 
\cite[(5.8)\,Proposition, Chapter V\hspace{-.1em}I]{Dem-book}).  
In particular, the operator $B_{\delta}$ has the inverse operator $B_{\delta}^{-1}$. 
By the Cauchy-Schwarz inequality, 
we have 
$$
|\lla v, \{\dbar b_{i_{0}}\}\rra|^{2} \leq 
\lla B_{\delta}^{-1}\, \{\dbar b_{i_{0}}\}, \{\dbar b_{i_{0}}\}\rra 
\ \lla B_{\delta}\, v, v \rra. 
$$

We first consider the limit of $\lla B_{\delta}\, v, v \rra$ 
when $\delta$ goes to zero. 
By Bochner-Kodaira-Nakano identity and the assumption of 
$\sqrt{-1}\Theta_{h}(E) \geq_{\rm{Nak}} 
f^{*}\sigma \otimes {\rm{id}}_{E}$, 
we have 
\begin{align*}
\lim_{\delta \to 0}\lla B_{\delta}\, v, v \rra
&=\lla f^{*}\sigma \Lambda_{\omega}\, v, v \rra \\
&\leq \lla \sqrt{-1}\Theta_{h}(E) \Lambda_{\omega}\, v, v \rra \\
&\leq \| \dbar v \|^{2} + \| \dbar^{*} v \|^{2}. 
\end{align*}
Now we prove that the norm 
$\lla B_{\delta}^{-1}\, \{\dbar b_{i_{0}}\}, 
\{\dbar b_{i_{0}}\}\rra_{X, h, \omega}$ on $X$
is uniformly bounded with respect to $\delta>0$  
if we suitably choose an increasing convex function $\chi$. 
For an arbitrary relatively compact set $K \Subset X$, 
we consider the norm $\lla B_{\delta}^{-1}\, \{\dbar b_{i_{0}}\}, 
\{\dbar b_{i_{0}}\}\rra_{K, h, \omega}$ on $K$. 
We can obtain  
\begin{align*}
\lla B_{\delta}^{-1}\, \{\dbar b_{i_{0}}\}, 
\{\dbar b_{i_{0}}\}\rra_{K, h, \omega}
&\leq 
\sum_{\substack{ i_{0} \in I   {\text{ with }} \\
U_{i_{0}} \cap K \not = \phi}}
\lla B_{\delta}^{-1}\, \dbar b_{i_{0}}, 
\dbar b_{i_{0}}\rra_{B_{i_{0}}, h, \omega}\\
&\leq \sum_{\substack{ i_{0} \in I   {\text{ with }} \\
U_{i_{0}} \cap K \not = \phi}}
\sum_{\substack{ j \in I   {\text{ with }} \\
U_{j} \cap U_{i_{0}} \not = \phi}}
\lla B_{\delta}^{-1}\, \dbar f^{*}\rho_{j} \wedge \dbar b_{ji_{0}}, 
\dbar f^{*}\rho_{j} \wedge \dbar b_{ji_{0}} \rra_{B_{i_{0}}, h, \omega}\\
&\leq D \sum_{\substack{ i_{0} \in I   {\text{ with }} \\
U_{i_{0}} \cap K \not = \phi}}
\sum_{\substack{ j \in I   {\text{ with }} \\
U_{j} \cap U_{i_{0}} \not = \phi}} 
\int_{B_{i_{0}}} |b_{ji_{0}}|^{2}_{h, \omega}\, |\dbar \rho_{j}|^{2}_{\sigma}
\,dV_{\omega}
\end{align*}
from Lemma \ref{par} and \ref{comp}. 
The proof of Lemma \ref{comp} is given at the end of this section. 
The right hand side does not depend on $\delta$ 
since the constant $D$ in Lemma \ref{comp} is independent of $\delta$. 
Moreover we may assume that 
it becomes finite if we choose rapidly increasing function $\chi$.  
Therefore 
we obtain 
$$
|\lla v, \{\dbar b_{i_{0}}\}\rra|^{2} \leq 
C (\| \dbar v \|^{2} + \| \dbar^{*} v \|^{2})
$$
for compactly supported smooth $v$.  
This completes the proof.

It remains to prove the following lemma\,: 
\begin{lemm}\label{comp}
Let $\varphi$ be an $E$-valued $(n,k)$-form on $X$ and 
$\rho$ is a smooth function on $Y$. 
Then there exists a positive constant $D$ 
$($depending only on $k$, ${\rm{rank}}\, E$$)$ 
such that 
$$
\big\langle B_{\delta}^{-1} (\dbar f^{*}\rho \wedge \varphi), 
(\dbar f^{*}\rho \wedge \varphi) \big\rangle_{h, \omega}
\leq D\, |\varphi|^{2}_{h, \omega}\, |\dbar \rho|^{2}_{\sigma}. 
$$
\end{lemm}
\begin{proof}
For a given point $x \in X$, 
we choose a local coordinate $z:=(z_{1},z_{2},\dots, z_{n})$ centered at $x$ 
such that 
\begin{align*}
\omega = \frac{\sqrt{-1}}{2} \sum_{j=1}^{n} 
 dz_{j} \wedge d\overline{z_{j}} \quad \text{and} \quad 
f^{*}\sigma = \frac{\sqrt{-1}}{2} \sum_{j=1}^{n} 
\lambda_{j} dz_{j} \wedge d\overline{z_{j}}
\quad  {\rm{ at}}\ x. 
\end{align*}
By taking a local embedding $i \colon U_{y} \hookrightarrow \mathbb{C}^{m}$
of a neighborhood $U_{y}$ of $y:=f(x) \in Y$,  
we can assume that $Y$ is an open set in $\mathbb{C}^{m}$
since $\rho $ (resp. $\sigma$) can be seen as the pull-back 
of a smooth function (resp. a K\"ahler form) on $\mathbb{C}^{m}$. 
We choose a local coordinate $t:=(t_{1},t_{2},\dots, t_{m})$ 
centered at $y$ such that 
\begin{align*}
\sigma = \frac{\sqrt{-1}}{2} \sum_{j=1}^{m} 
 dt_{j} \wedge d\overline{t_{j}} 
\quad  {\rm{ at}}\ y=f(x). 
\end{align*}
When we write $f=(f_{1}, f_{2}, \dots, f_{m})$ with respect to these coordinates, 
we have 
$$
\sum_{p=1}^{m} \frac{\partial f_{p}}{\partial z_{k}}
\overline{\frac{\partial f_{p}}{\partial z_{\ell}}}= \delta_{k \ell} \lambda_{k}
\quad  {\rm{ at}}\ x. 
$$
For the local expressions with respect to these coordinates 
$$
\varphi=\sum_{I} \sum_{i=1}^{{\rm{rank}\,E}} \varphi_{I,i}(z)\, 
e_{i}\otimes dz \wedge d\overline{z_{I}} 
\quad \text{and} \quad 
\dbar f^{*}\rho= \sum_{k=1}^{n} \sum_{j=1}^{m}
\frac{\partial \rho}{\partial \overline{t}_{j}}(f(z))
\overline{\frac{\partial f_{j}}{\partial z_{k}}}(z) 
d\overline{z}_{k}, 
$$
we can easily see that 
$$
|\varphi|^{2}_{\omega}=\sum_{I} \sum_{i=1}^{{\rm{rank}\,E}} |\varphi_{I,i}|^{2} 
\ \ \text{ at } x
\quad \text{and} \quad 
|\dbar \rho|^{2}_{\sigma}=\sum_{j=1}^{m} 
\big| \frac{\partial \rho}{\partial \overline{t}_{j}}(f(0)) \big|^{2} 
\quad \text{ at } y. 
$$
Here $dz$ is $dz:=dz_{1}\wedge dz_{2}\cdots\wedge dz_{n}$, 
$d\overline{z_{I}}$ is
$d\overline{z_{I}}:=d\overline{z_{i_{1}}}\wedge d\overline{z_{i_{2}}}\cdots\wedge 
d\overline{z_{i_{k}}}$ 
for an ordered multi-index $I=\{i_{1}<i_{2}<\dots<i_{k}\}$, and  
$e_{i}$ is a local frame of $E$ that gives an orthonormal basis 
at $x$.
Then, by putting
$$
g_{k}(z):=
\sum_{j=1}^{m}\frac{\partial \rho}{\partial \overline{t}_{j}}(f(z))
\overline{\frac{\partial f_{j}}{\partial z_{k}}}(z), 
$$
we have 
\begin{align*}
|g_{k}(0)|^{2} 
&\leq \, |\dbar \rho|^{2}_{\sigma}\, \Big( \sum_{p=1}^{m}\big|
\frac{\partial f_{p}}{\partial z_{k}}(0)\big|^{2} \Big)
\end{align*}
by the Cauchy-Schwarz inequality. 
Let $\mu_{j}$ be an eigenvalue of $\omega_{\delta}$ 
with respect to $\omega$. 
We remark that $\mu_{j}=\delta + \lambda_{j}$ at $x$. 
Then, by straightforward computations, 
we obtain   
\begin{align*}
&\big\langle B_{\delta}^{-1} (\dbar f^{*}\rho \wedge \varphi), 
(\dbar f^{*}\rho \wedge \varphi) \big\rangle_{h, \omega} \\
=&  \sum_{J}\sum_{i=1}^{{\rm{rank}}\, E}
\Big( \sum_{{\text{$I\cup \{k\}=J$}}} 
{\rm{sgn}}
\begin{pmatrix}
J
\\
Ik
\end{pmatrix}
\varphi_{I,i}\, 
g_{k}(0)\Big)
\Big( \sum_{{\text{$I'\cup \{k'\}=J$}}} 
{\rm{sgn}}
\begin{pmatrix}
J
\\
I'k'
\end{pmatrix}
\varphi_{I',i}\, 
g_{k'}(0) \Big)
\Big( {\sum_{j \in J} \mu_{j}} \Big)^{-1}\\
\leq& 
C_{1} \sum_{J}\sum_{i=1}^{{\rm{rank}}\, E} 
\sum_{{\text{$I\cup \{k\}=J$}}}
|\varphi_{I,i}|^{2}\, |g_{k}(0)|^{2}\, 
\Big( {\sum_{j \in J} \mu_{j}} \Big)^{-1}\\
\leq& 
C_{2}\, |\varphi|^{2}_{\omega}\, 
|\dbar \rho|^{2}_{\sigma}
\sum_{J}
\frac{\sum_{k \in J}\sum_{p=1}^{m}\big|
\frac{\partial f_{p}}{\partial z_{k}}(0)\big|^{2}}
{\sum_{j \in J}(\delta + \sum_{p=1}^{m} |\frac{\partial f_{p}}{\partial z_{j}}(0)|^{2})}
\end{align*} 
for some constants $C_{1}, C_{2}>0$.  
The first inequality follows from the fundamental inequality 
$(\sum_{i=1}^{N}|a_{i}|)^{2} \leq 2^{N-1} 
\sum_{i=1}^{N}|a_{i}|^{2}$ and the second inequality 
follows from the  above estimate for $g_{k}(0)$. 
The last term is smaller than or equal to one for any $\delta>0$. 
This completes the proof. 
\end{proof}

\begin{rem}\label{compare}
(1) 
In his paper \cite{Ohs84}, 
Ohsawa proved Theorem \ref{m2} in the case $q=0$, 
whose strategy is as follows: 
He showed that, for a $\dbar$-closed $E$-valued $(n,k)$-form $f$ on 
a complete K\"ahler manifold, 
there exists a solution $g$ of the $\dbar$-equation 
$\dbar g = f$ satisfying $\| g \|_{\sigma}\leq C \|f\|_{\sigma}$ 
for some positive constant $C$  
if $\|f\|_{\sigma}< \infty$. 
Here $\| \bullet \|_{\sigma}$ denotes the $L^{2}$-norm define by 
$\| \bullet \|_{\sigma} :=\lim_{\delta \to 0} \|\bullet\|_{\delta \omega + \sigma}$. 
(See \cite[Theorem 2.8]{Ohs84} for the precise statement.)
From this celebrated result, he obtained Theorem \ref{m2} 
by inductively constructing solutions of the $\dbar$-equation. 
\vspace{0.1cm}\\
(2) 
However, we can not expect $\|\dbar b_{i_{0}}\|_{\sigma} < \infty$ 
in the case $q>0$. 
In fact, we can prove that $\|\dbar b_{i_{0}}\|_{\sigma} =\infty$ in this case,  
by using \cite[Theorem 5.2]{Tak95}. 
For this reason, we estimate the limit of $ \lla B^{-1}_{\delta}\, 
\{ \dbar b_{i_{0}} \}, \{ \dbar b_{i_{0}} \} \rra$ 
instead of $\|\dbar b_{i_{0}}\|_{\sigma}$. 
\vspace{0.1cm}\\
(3) 
By the same computation as the proof of Lemma \ref{comp}, 
we can check that 
$\lim_{\delta \to 0}\lla B^{-(p+q)}_{\delta}\, \{ \dbar b_{i_{0}} \}, 
\{ \dbar b_{i_{0}} \}\rra$ is finite. 
If we can construct a solution $g$ 
of the $\dbar$-equation $\dbar g = f$ 
such that $\lim_{\delta \to 0} \lla B^{-(k-1)}_{\delta}\, g, g \rra < \infty$
from the assumption of 
$\lim_{\delta \to 0}\lla B^{-k}_{\delta}\, f, f \rra < \infty$,  
we can prove Theorem \ref{m2} by the same method as in \cite{Ohs84}. 
However this strategy did not succeed, 
and thus we need the injective map $\varphi_{p,q}$ constructed in Theorem \ref{m1}. 
\end{rem}


\newpage

\end{document}